\documentclass{amsart}
\usepackage{amsmath, amscd, amssymb, amsthm}
\usepackage{sbbm}
\usepackage{latexsym}
\usepackage{amsfonts}
\usepackage{graphicx}
\usepackage[all,cmtip]{xy}

\usepackage[
dvipdfm, colorlinks=true, linkcolor=black,breaklinks=true,
urlcolor=blue, citecolor=black
]{hyperref}%
\usepackage{geometry}
\newtheorem{theorem}{Theorem}[section]
\newtheorem{lemma}[theorem]{Lemma}
\newtheorem{corollary}[theorem]{Corollary}

\newtheorem{question}[theorem]{Question}

\newcommand{\R}{{\mathbb R}}

\newcommand{\be}{\begin{equation}}
\newcommand{\ee}{\end{equation}}
\newcommand{\bea}{\begin{eqnarray}}
\newcommand{\eea}{\end{eqnarray}}

\newcommand{\vs}{\vspace{0.5cm}}

\newcommand{\vsv}{\vspace{0.12cm}}

\def\XXint#1#2#3{{\setbox0=\hbox{$#1{#2#3}{\int}$ }
\vcenter{\hbox{$#2#3$ }}\kern-.6\wd0}}

\begin{document}

\title{Lie Groups with flat Gauduchon connections}

\author{Luigi Vezzoni}

\address{Luigi Vezzoni. Dipartimento di Matematica G. Peano, Universit\.{a} Di Torino, Via Carlo Alberto 10, 10123 Torino, Italy.}
\email{{luigi.vezzoni@unito.it}}

\author{Bo Yang}\thanks{The research of BY is partially supported by a grant from Xiamen University (0020-ZK1088).}

\address{Bo Yang. School of Mathematical Sciences, Xiamen University, Xiamen, Fujian, 361005, China.}
\email{{boyang@xmu.edu.cn}}

\author{Fangyang Zheng}\thanks{The research of FZ is partially supported by a Simons Collaboration Grant 355557.}

\address{Fangyang Zheng.  Department of Mathematics,
The Ohio State University, 231 West 18th Avenue, Columbus, OH 43210, USA}

%\address{Fangyang Zheng. Department of Mathematics, The Ohio State
%University, 231 West 18th Avenue, Columbus, OH 43210, USA}
%and Zhejiang Normal University, Jinhua, 321004, Zhejiang, China

\email{{zheng.31@osu.edu}}

\vsv



\vsv

\begin{abstract}
We pursuit the research line proposed in \cite{YZ-Gflat} about the
classification of Hermitian manifolds whose $s$-Gauduchon connection
$\nabla^s =(1-\frac{s}{2})\nabla^c + \frac{s}{2}\nabla^b$ is flat,
where $s\in \R$ and  $\nabla^c$ and $\nabla^b$ are the Chern and the
Bismut connections, respectively. We focus on Lie groups equipped
with a left invariant Hermitian structure. Such spaces provide an
important class of Hermitian manifolds in various contexts and are
often a valuable vehicle for testing new phenomena in complex and
Hermitian geometry. More precisely, we consider a connected
$2n$-dimensional Lie group $G$ equipped with a left-invariant
complex structure $J$ and a left-invariant compatible metric $g$ and
we assume that its connection $\nabla^s$ is flat. Our main result
states that if either $n$=2 or there exits a $\nabla^s$-parallel
left invariant frame on $G$, then $g$ must be K\"ahler.
%In the case of $n=2$, we give a complete description of the group structure by applying some previous results on Lie groups admitting a flat left invariant K\"ahler metric. %In the latter case, the Lie group turns out to be abelian.
This result demonstrates rigidity properties of some complete
Hermitian manifolds with  $\nabla^s$-flat Hermitian metrics.
\end{abstract}

\maketitle

\tableofcontents

\markleft{Gauduchon flat Lie Groups}

\markright{Gauduchon flat Lie Groups}

\section{Introduction}

An affine connection $\nabla$ on a Hermitian manifold $(M, J, g)$ is
called {\em Hermitian} if preserves the Hermitian structure, i.e.
$\nabla J=\nabla g=0$. Every Hermitian manifold admits Hermitian
connections and canonical choices of such connections can be done by
imposing some constraints on the torsion. For instance the {\em
Chern connection} $\nabla^c$ (see e.g. Proposition 10.2 on p.180 in
\cite{KN}) and the {\em Bismut connection} $\nabla^b$ (see
\cite{Bismut} and \cite{Yano}) are canonical Hermitian connections.
A precise definition of {\em canonical Hermitian connection} was
given by Gauduchon in \cite{Gauduchon2} who also proved that the
space of canonical Hermitian connections can be parametrized  as
\[
\{\nabla^s=(1-\tfrac{s}{2})\nabla^c + \tfrac{s}{2}\nabla^b\,\,:\,\,
s\in \mathbb{R}\}.\] In the following given a real number $s$, we
call $\nabla^s$ as the {\em $s$-Gauduchon connection}. Note that
$\nabla^0=\nabla^c$ and $\nabla^2=\nabla^b$. Moreover, $\nabla^1$
corresponds to the projection of the Levi-Civita connection onto the
space of $(1,0)$-vector fields. This connection is called the {\em
first canonical Hermitian connection} in \cite{Gauduchon2} and
\cite{L1957}, the {\em associated connection} in \cite{GK} and the
{\em complexified Levi-Civita connection} in \cite{LiuYang}. In
\cite{YZ-Gflat}, it is proposed the study of compact Hermitian
manifolds with flat $s$-Gauduchon connections, as a special case of
a problem posed by Yau (Problem 87 in \cite{Yau}). The case $s=0$
(i.e. when the Chern connection is flat) was studied in the 1950s by
W. Boothby \cite{Boothby}  and H. C. Wang \cite{Wang}. Boothby
proved that if $(M^n,g)$ is a compact Chern-flat Hermitian manifold,
then its universal cover $\widetilde{M}$ has a natural structure of
complex Lie group and $g$ lifts to a left invariant metric on
$\widetilde{M}$. In this result the compactness assumption cannot be
removed in general, since Boothby constructed local examples of
Chern-flat Hermitian surfaces with non-parallel torsion.
%thus  $\widetilde{M}$ cannot be a Lie group equipped with a left invariant metric.
In Section 5 of \cite{WYZ}, it was given a complete noncompact example of such kind.

For $s=2$, Bismut-flat Hermitian manifolds were classified in \cite{WYZ} (even in the non-compact case). From \cite{WYZ} it follows that the universal cover of a Bismut-flat Hermitian manifold is an open subset of a Samelson space, namely, a simply-connected, connected Lie group $G$ equipped with a bi-invariant metric and a compatible left invariant complex structure. Note that any Samelson space is the product of a compact semi-simple Lie group with a vector group by Milnor's lemma in \cite{Milnor}.

Also $\nabla^1$-flat manifolds have been studied in several papers, and,
even though there is no exhaustive classification of such manifolds, there
are some partial classification results
(see \cite{Andrada2012} and \cite{GK}).

In  \cite{YZ-Gflat} it was conjectured that a compact $\nabla^s$-flat Hermitian manifold, for $s\neq 0, 2\,$, must be K\"ahler, and thus a finite undercover of a flat complex torus. The conjecture was confirmed in \cite{YZ-Gflat} for Hermitian surfaces, as well as in general dimensions under the additional assumption that $s\notin [a_n^-,a_n^+]$. Here $a_n^{\pm}$ are explicit constants with $a_n^- > 0.53$ and $a_n^+ < 7.47$.

In this paper we focus on connected Lie groups equipped with a left invariant Hermitian structure. Such manifolds (and their quotients) costituite a rich and interesting class of Hermitian manifolds. Classical examples are provided by even dimensional compact Lie groups, since every even-dimensional compact Lie group admits a left invariant complex structure compatible with its bi-invariant metric (see \cite{Samelson} and \cite{WangAJM}). Moreover, left invariant Hermitian structures exist on some nilpotent and solvable Lie groups, see e.g \cite{Snow}, \cite{CFGU}, \cite{Salamon2001}, \cite{Salamon2002} and the reference therein.

%The main purpose of this paper is to prove rigidity of such a class with a flat Gauduchon connection. At the same time, we might expect the existence of interesting examples
%in such a class with a flat Hermitian connection outside the Yano-Gauduchon line. However, it is
%unclear to us how to propose a reasonable question in the latter case.

The main purpose of this paper is to investigate the following question:

\begin{question} \label{main question}
Let $G$ be a connected Lie group equipped with a left invariant metric $g$ and a compatible left invariant complex structure $J$. Assume that the connection $\nabla^s$ of $(G,g)$ is flat for some $s\neq 0,2$. Then must $g$ be K\"ahler?
\end{question}

From Theorem 4.14 in \cite{Andrada2012} it follows that  Question \ref{main question} has an affirmative answer when $s=1$ and the complex structure is abelian, i.e. when the space of left invariant $(1,0)$-vector fields on $G$ is an abelian Lie algebra.

The first main results of the present paper is that Question \ref{main question} has an affirmative answer in real dimension 4:

%Note that our previous work \cite{YZ-Gflat} is concerned with compact Hermitian manifolds with flat $\nabla^s$-connection, and our methods and results there can not apply to Question \ref{main question} directly. In any case, we are able
%to answer Question \ref{main question} affirmatively in the case of real dimension $4$.

\begin{theorem}\label{intro1}
Let $G$ be a connected Lie group of real dimension $4$ equipped with a left invariant metric $g$ and a compatible left invariant complex structure $J$. Assume that the Hermitian manifold $(G,g)$ has a flat $s$-Gauduchon connection $\nabla^s$ for some $s\neq 0, 2$. Then $g$ is K\"ahler.
\end{theorem}

For general dimensions we give an affirmative answer to Question \ref{main question} under the extra assumption on $(G,g)$ to admit a left-invariant $\nabla^s$-parallel frame.

\begin{theorem}\label{intro2}
Let $G$ be a connected Lie group of real dimension $2n$ equipped with a left invariant metric $g$ and a compatible left invariant complex structure $J$. Assume that $(G,  g)$ has a flat $s$-Gauduchon connection $\nabla^s$ for some $s\neq 0, 2$. If $G$ admits a left invariant $\nabla^s$-parallel tangent frame, then $g$ is K\"ahler. Moreover, $G$ is abelian.
\end{theorem}

Under any left invariant frame, the torsion components of a left invariant connection are always constant ((\ref{torsion formula}) in Section 2). In particular, if a left invariant connection admits a left invariant parallel frame, then its torsion is parallel. By the result of Kamber and Tondeur \cite{KT} (see also \cite{Andrada2012}), given any smooth manifold $M$ equipped with a complete affine connection $\nabla$ which is flat and whose torsion is  parallel, then the universal cover of $M$ is a Lie group and $\nabla$ lifts to a left invariant connection. Hence theorem \ref{intro2} can be restated as follows:

\begin{theorem}\label{intro3}
Let $(M^n,g)$ be a complete Hermitian manifold with a flat $s$-Gauduchon connection $\nabla^s$ and with parallel torsion. Assume that $s\neq 0,2$, then its universal cover is holomorphically isometric to the complex vector group $\mathbb{C}^n$.
\end{theorem}

As pointed out at the beginning of Section 4, a sufficient condition for the existence of a parallel left invariant  frame is that Lie group does not admit any non-trivial finite dimensional unitary representation. Hence, as a corollary of Theorem \ref{intro2} we have:

\begin{corollary}
Let $G$ be a connected Lie group equipped with a left invariant metric $g$ and a compatible left invariant complex structure $J$. Assume that the Hermitian manifold $(G,g)$ has a flat $s$-Gauduchon connection $\nabla^s$ and $s\neq 0, 2$. If $G$ is non-compact and contains no
non-trivial proper normal closed subgroup (or more generally, if $G$ admits no non-trivial finite dimensional unitary representation), then $g$ is K\"ahler and $G$ is abelian.
\end{corollary}

In order to prove our theorems, we show as preliminary results some algebraic relations on Hermitian Lie groups involving the structure constants of the Lie group, the coefficients of the Gauduchon connections and the coefficients of the Chern connection.

It is quite natural to wonder if the extra assumption in Theorem \ref{intro2} about the existence of a left invariant parallel frame could be removed.
%
%Next we explain the strategy we use to prove our results. Let $(G, J, g)$ of real $2n$-dimensional Lie group with a left invariant Hermitian structure and let $\{e_1, \cdots, e_n\}$ be a unitary left invariant frame. The structure equations of $G$ can be written as
%\[
%C_{ik}^j = g( [e_i, e_k], \overline{e}_j), \ \ \  D^j_{ik}= g( [\overline{e}_j, e_k], \,e_i)
%\] and the coefficients of $\nabla^{s}$ are given by
%$$
%\Gamma_{ik}^j =g( \nabla^s_{e_k} e_i , \, \overline{e}_j)\,.
%$$
%Let $T^{c}$ denote the torsion tensor of the Chern connection $\nabla^{0}$.
%%Then Theorem 4.14 in \cite{Andrada2012} implies if $\nabla^{1}$ is flat
%%and $C=0$, then $D=0$ and $T^c=0$.
%In the proof of Theorem \ref{intro1} we show that if $\nabla^{s}$ is flat, for some $s\neq 0,2$ and $n=2$ then $T^c=0$, and Theorem \ref{intro1} implies if $\nabla^{s} (s \neq 0, 2)$ is flat and $\Gamma_{ik}^j=0$, then $C=D=T^c=0$. Note that $\Gamma_{ik}^j=0$ follows from the existence of left invariant parallel tangent frame by (\ref{P formula}) in Section 4.
%
%
Though we are working on left invariant geometry and the calculation is of local nature, it seems quite complicated to overcome some algebraic complexities arising in our computations
(see the observation after Lemma \ref{algebriac relation}), so we believe some new insights are needed to tackle it.

Note that if Question \ref{main question} has an affirmative answer,
i.e. the universal cover of $(G,g)$ is $\mathbb{C}^n$ equipped with
the Euclidean metric, it is not necessarily true that the
corresponding Lie group is the complex vector group. In fact, there
are examples of non-abelian Lie groups which admits left invariant
flat K\"ahler structures, see e.g. \cite{BDF} and the appendix at
the end of the present paper. The results in \cite{BDF} give a
K\"ahler analogue of Milnor's classification theorem on flat Lie
groups with left-invariant Riemannian metrics \cite{Milnor}. As an
application we have the following improvement of Theorem
\ref{intro1}:

\begin{corollary} \label{flat Kahler intro}
Let $(G, J, g)$ be as in  the statement of Theorem \ref{intro1}. Then either $G$ is isomorphic to a complex vector group or its
Lie algebra $\mathfrak{g}$ admits the following orthogonal decomposition:
\[
\mathfrak{g}=\mathfrak{c} \oplus \mathfrak{k}, \ \text{where the subalgebras} \ \mathfrak {c}=\operatorname{span}\{W\} , \text{and}\ \mathfrak {k}=\operatorname{span}\{X, Y, Z\}.
\]
Here $\{X, Y, Z, W\}$ is an orthonormal basis of $\mathfrak {g}$,
$\mathfrak {c}$ is the center of $\mathfrak{g}$, $[X, Y]=qZ$, $[X,
Z]=-qY$, $[Y,Z]=0$ where $q \neq 0$ is any given real number, and
the complex structure $J$ is defined by $JX=W$ and $JY=Z$.
\end{corollary}

Finally, we would like to point out some interesting recent works relating left invariant geometry of complex Lie groups and Gauduchon connections: Fei-Yau found in \cite{FeiYau} invariant solutions to the Hull-Strominger system on unimodular complex Lie groups and Phong-Picard-Zhang \cite{PPZ} studied solutions to the Anomaly Flow on unimodular complex $3$-dimensional Lie groups.

\vs

\section{Lie groups as Hermitian manifolds} \label{surface}

\vsv

Let $G$ be a connected Lie group of real dimension $2n$, and denote by ${\mathfrak g}$ its Lie algebra. Let $g$ be a left invariant metric on $G$ and $J$ a left invariant complex structure on $G$ that is compatible with $g$. Note that the Hermitian manifold $(G,g)$ of complex dimension $n$ is always complete.

Left invariant metrics $g$ on $G$ are in one one correspondence with inner products $\langle \, , \rangle$ on ${\mathfrak g}$, while (compatible) left invariant complex structures on $G$ are in one one correspondence with almost complex structures $J$ on the vector space ${\mathfrak g}$ (compatible with $\langle \, , \rangle$) satisfying the integrability condition
\begin{equation*}
[x,y] - [Jx,Jy] + J[Jx,y] + J[x,Jy] = 0, \ \ \forall \ x,y \in {\mathfrak g},
\end{equation*}
or equivalently,
\begin{equation}
[{\mathfrak g}^{1,0}, \, {\mathfrak g}^{1,0} ] \subseteq {\mathfrak g}^{1,0},
\end{equation}
where ${\mathfrak g}^{1,0}$ is the $\sqrt{\!-\!1}$-eigenspace of $J$ in ${\mathfrak g}^{\mathbb C}={\mathfrak g} \otimes {\mathbb C}$.

Let $\{ e_1, \ldots , e_n\}$ be a unitary basis of ${\mathfrak g}^{1,0}$. We will also use $e_i$ to denote the corresponding left invariant vector field on $G$, so it becomes a global unitary frame of type $(1,0)$ tangent fields on the complete Hermitian manifold $(G,g)$.  The integrability condition can be expressed as
\begin{equation}
\langle [e_i,e_j],\, e_k\rangle = 0 , \ \ \ \ \forall  \ 1\leq i, j, k \leq n.
\end{equation}
Let $\nabla$ be the Levi-Civita connection. Since $\nabla$ is torsion free, we have
\begin{equation}
2\langle \nabla_{x}y, z\rangle = x\langle y, z\rangle + y\langle x, z\rangle - z\langle x, y\rangle +\langle [x,y],\, z\rangle -\langle [x,z], \,y\rangle - \langle [y,z], \,x\rangle
\end{equation}
for any vector fields $x$, $y$, $z$ on $G$. Let us write $e=\, ^t\!(e_1, \ldots , e_n)$, and let $\{ \varphi_1, \ldots , \varphi_n\}$ be the coframe of $(1,0)$-forms dual to $e$. Denote by $T^k_{ij}$  the components of the Chern torsion under the frame $e$, namely,  $ T^c(e_i,e_j) = 2\sum_{k=1}^n T_{ij}^k e_k $. Following the notations of  \cite{YZ} and \cite{YZ-Gflat}, we have
\begin{eqnarray*}
\nabla e & = &  \nabla^1e + \beta e = \nabla^c e + \gamma e + \beta e,\\
\beta e_i & = &  T_{ij}^k \, \overline{\varphi}_k \overline{e}_j, \\
\gamma e_i & = & (T_{ik}^j \varphi_k - \overline{T^i_{jk}} \, \overline{\varphi}_k) e_j.
\end{eqnarray*}
So by the formula for $\nabla$ we get
\begin{equation}
2 T_{ij}^k = 2\langle \beta(\overline{e}_k) e_i, e_j\rangle =  \langle [\overline{e}_k, e_i], \, e_j \rangle  - \langle [\overline{e}_k,  e_j], \, e_i \rangle  - \langle [e_i, e_j], \, \overline{e}_k  \rangle. \label{torsion}
\end{equation}
From this, a direct computation leads to
\begin{eqnarray}
\langle \nabla^c_{e_k}e_i , \, \overline{e}_j\rangle  & = &  \langle [\overline{e}_j, e_k], \, e_i \rangle \\
\nabla^c_{\overline{e}_k}e_i & = & [\overline{e}_k, e_i]^{1,0},\\
2 \langle \gamma (e_k) e_i,\, \overline{e}_j\rangle & = & 2T_{ik}^j \ = \ \langle [\overline{e}_j, e_i], \,e_k \rangle  - \langle [\overline{e}_j,  e_k], \, e_i \rangle  - \langle [e_i, e_k],\, \overline{e}_j  \rangle, \\
\langle \nabla^s_{e_k}e_i , \, \overline{e}_j\rangle  & = &  (1-\frac{s}{2})  \langle [\overline{e}_j, e_k], \,e_i \rangle  +\frac{s}{2}  \langle [\overline{e}_j, e_i], \,e_k \rangle - \frac{s}{2} \langle [e_i, e_k],\, \overline{e}_j  \rangle, \\
\langle \nabla^s_{\overline{e}_k} e_i , \, \overline{e}_j\rangle  & = &  (1-\frac{s}{2})  \langle [\overline{e}_k, e_i], \,\overline{e}_j \rangle  +\frac{s}{2}  \langle [\overline{e}_j, e_i], \, \overline{e}_k \rangle + \frac{s}{2} \langle [\overline{e}_j, \overline{e}_k],\, e_i  \rangle, \\
\langle \nabla^b_{e_k}e_i , \, \overline{e}_j\rangle  & = &   \langle [\overline{e}_j, e_i], \,e_k \rangle -  \langle [e_i, e_k],\, \overline{e}_j  \rangle.  \label{b-connection}
\end{eqnarray}
Here we used the fact that $\nabla^s=\nabla^c+s\gamma$, and $\nabla^c=\nabla^0$, $\nabla^b=\nabla^2$. In the following, let us denote by
$$ C_{ik}^j = \langle [e_i, e_k], \overline{e}_j\rangle , \ \ \  D^j_{ik}= \langle [\overline{e}_j, e_k], \,e_i \rangle$$
 for the structure constants of ${\mathfrak g}$, and denote by  $\Gamma_{ik}^j = \langle \nabla^s_{e_k} e_i , \, \overline{e}_j\rangle$ the connection components of $\nabla^s$ under the frame $e$. Then we have
\begin{eqnarray}
\Gamma_{ik}^j  & = &  D^j_{ik} + sT^j_{ik} \\
2T_{ik}^j & = & -D_{ik}^j + D_{ki}^j -C_{ik}^j   \label{torsion formula}\\
  T^j_{ik,\ell }  & = &  \sum_{r=1}^n \big(\!-T^j_{rk}\Gamma^r_{i\ell} - T^j_{ir}\Gamma^r_{k\ell} + T^r_{ik}\Gamma^j_{r\ell} \big) \\
   T^j_{ik,\overline{\ell }}  & = &  \sum_{r=1}^n \big( T^j_{rk}\overline{\Gamma^i_{r\ell}} + T^j_{ir}  \overline{\Gamma^k_{r\ell}} - T^r_{ik} \overline{\Gamma^r_{j\ell}} \big)
\end{eqnarray}
for any indices $i$, $j$, $k$, and $\ell$.  Here the indices appearing after the ``," stand for covariant differentiation with respect to the connection $\nabla^s$.

Since $ [\overline{e}_j,e_i] = \sum_{k} \big( D^j_{ki}\overline{e}_k - \overline{D^i_{kj}} e_k \big)$, and the components of $C$, $D$, $T$,  $\Gamma$ are all constants, the Jacobi identity for $\{ e_i,e_j,e_k\}$ and for $\{ e_i, \overline{e}_j, e_k\}$ will give us the following identities

\begin{lemma}
Let $G$ be a Lie group equipped with a left invariant metric $g$ and a compatible left invariant complex structure. Let $e$ be a left invariant unitary frame on $G$, and denote by $C$, $D$ the structure constants. The Jacobi identities are
\begin{eqnarray}
&& \sum_{r=1}^n \big( C^r_{ij}C^{\ell}_{rk} + C^r_{jk}C^{\ell}_{ri} + C^r_{ki}C^{\ell}_{rj} \big) \ =\ 0 \\
 && \sum_{r=1}^n \big(  C^r_{ik} D^{\ell}_{j r} + D^r_{ji} D^{\ell}_{rk} - D^r_{jk} D^{\ell}_{ri}   \big) \ =\ 0 \\
 && \sum_{r=1}^n \big(   C^r_{ik} \overline{D^r_{j\ell }} - C^{j}_{rk} \overline{D^i_{r\ell}} + C^{j}_{ri} \overline{D^k_{r \ell}} - D^{\ell}_{ri} \overline{D^k_{j r}} + D^{\ell}_{rk} \overline{D^i_{j r}}  \big) \ =\ 0 \label{eq:thirdJacobi}
\end{eqnarray}
for any indices $i$, $j$, $k$, and $\ell$.
\end{lemma}

\vs

\section{Hermitian manifolds with flat $s$\,-Gauduchon connection}

\vsv

Let $(M^n,g)$ be a Hermitian manifold with flat $s$-Gauduchon connection $\nabla^s$. Under any local unitary tangent frame $\{ E_1, \ldots , E_n\}$, let us denote by $T^k_{ij}$ the components of the Chern torsion under $E$:
$$ T^c(E_i, E_j) = \sum_{k=1}^n 2T^k_{ij} E_k.$$
We first prove the following:

\begin{lemma}
Let $(M^n,g)$ be a $\nabla^s$-flat Hermitian manifold, with $s\neq 0$. Then under any unitary frame $E$, the components of the Chern torsion $T^c$ and its $\nabla^s$-covariant differentiation will  satisfy the following:
\begin{eqnarray}
   T^{\ell}_{ij,k} - T^{\ell}_{ik,j} & = & 2(1-s) \, T^{\ell }_{ir} T^r_{jk} + s \, T^{\ell }_{jr} T^r_{ik} - s \, T^{\ell }_{kr} T^r_{ij} \\
 0 & = &   (n-2)(s-1) \, \big( T^{\ell }_{ir} T^r_{jk} +  T^{\ell }_{jr} T^r_{ki} + T^{\ell }_{kr} T^r_{ij} \big)  \\
 T^{\ell}_{ij,k} - T^{\ell}_{ik,j} & = & (2-s)\, T^{\ell }_{ir} T^r_{jk}   \ \ \ \ \ \ \  \mbox{if} \ \ n\geq 3  \ \ \mbox{and} \ \ s\neq 1\\
4(s-1)(2s-1)\,  T^k_{ij, \,\overline{\ell }} & = & -4s(s-1)^2 \, T^r_{ij} \overline{T^r_{k\ell }} -s\, (5s^2-10s+4) \, \big(  T^k_{ir} \overline{ T^j_{\ell r}} - T^k_{jr} \overline{ T^i_{\ell r}}  \big) \\
\nonumber  & & +\, s^3\, \big(  T^{\ell}_{ir} \overline{ T^j_{k r}} - T^{\ell}_{jr} \overline{ T^i_{k r}} \big)
\end{eqnarray}
for any indices $i$, $j$, $k$, and $\ell$.
\end{lemma}

\begin{proof}
Note that the validity of these identities are independent of the choice of unitary frames here as both sides representing tensors. So we may prove them under a local unitary frame $E$ that is $\nabla^s$-parallel, which exists since $g$ is $\nabla^s$-flat. The first three identities are directly from Lemma 3.1 and Lemma 3.2 in \cite{YZ-Gflat}. So we only need to prove the last identity. Let us denote by $X_{ij}^{k\ell}$ the right hand side of the formula (21) in Lemma 3.1 of \cite{YZ-Gflat}. Write
$$ u = X_{ij}^{k\ell} - X_{ij}^{\ell k}, \ \ \ v = \overline{X^{ij}_{k\ell} - X^{ji}_{k\ell} }, \ \ \ x= T^k_{ij, \,\overline{\ell }} - T^{\ell }_{ij, \,\overline{k }}, \ \ \ y = \overline{   T^i_{k\ell , \,\overline{j }} - T^{j }_{k\ell , \,\overline{i }} }  $$
Then we get from that formula (21) of \cite{YZ-Gflat} that
$$ 2(s-1)x+2sy =u, \ \ \ 2sx+2(s-1)y = v.$$
Since $u=v$, we get $2(2s-1)x=2(2s-1)y=u$. Plug this back into that formula (21), we obtain the last identity of the present lemma.
\end{proof}

In the last identity of Lemma 3.1, if we take $s=\frac{1}{2}$, and let $i=k$, $j=\ell$ and sum them over, we get
$$ 0 = \frac{1}{4} \sum_{i,j,\,r=1}^n |T^j_{ir}|^2 + \frac{1}{4} \sum_{r=1}^n |\eta_r|^2 $$
where $\eta_r = \sum_k T^k_{kr}$ are the components of the Gauduchon's torsion $1$-form. So when $s=\frac{1}{2}$, we always have $T=0$, that is, any $\nabla^{\frac{1}{2}}$-flat Hermitian manifold is K\"ahler. This is observed in Proposition 1.8 of \cite{YZ-Gflat}.

Next let us specialize to the $2$-dimensional case. Let $(M^2,g)$ be a $\nabla^s$-flat Hermitian surface with $s\neq 0$. At any point in $M^2$, by rotating the unitary frame $E$ if necessary, we may always assume that $T^2_{12}=0$ at this point. In fact, assuming that $g$ is not K\"ahler, then in the open subset of $M^2$ where $T$ is not identically zero, there exists local unitary frame $E$ such that $T^2_{12}=0$, and $T^1_{12}=\lambda \neq 0$. Under such a frame, by taking $i=1$, $j=2$ in the first and last identity of Lemma 3.1, we get the following:

\begin{lemma}
Let $(M^2,g)$ be a $\nabla^s$-flat Hermitian surface with $s\neq 0$. Suppose $E$ is a local unitary frame such that $T^2_{12}=0$. Then we have
\begin{eqnarray}
&& T^1_{12,1}=T^2_{12,1}=T^2_{12,2}=0, \ \ \ T^1_{12,2} = (2-s)\lambda^2 ,\\
&& 4(s-1)(2s-1) \,T^1_{12, \overline{1}} = 4(s-1)(2s-1) \,T^2_{12, \overline{2}} = 0 ,\\
&& 4(s-1)(2s-1) \,T^1_{12, \overline{2}} = s^2(s-2)|\lambda |^2 ,\\
&& 4(s-1)(2s-1) \,T^2_{12, \overline{1}} = s\,(3s^2-8s+4)|\lambda |^2,
\end{eqnarray}
where $\lambda = T^1_{12}$.
\end{lemma}

Note that when $g$ is not K\"ahler,  $\lambda $ is not identically zero, by the third identity above we know that  $s\neq \frac{1}{2}$ and $s\neq 1$, so all the covariant derivatives of the Chern torsion can be expressed in terms of $\lambda $ under this frame $E$. We will see in the next section that this cannot happen when $s\neq 2$ and when the Hermitian surface is a Lie group equipped with left invariant metric and left invariant complex structure, hence proving Theorem 1.2.

\vs

\section{Lie groups with flat $s\,$-Gauduchon connection}

\vsv

Now let us specialize to Hermitian manifolds $(G,g)$ which are Lie groups of real dimension $2n$ equipped with a left invariant metric $g$ and a compatible left invariant complex structure $J$. We will assume that $(G,g)$ is $\nabla^s$-flat.  Lifting to the universal cover if necessary, we may assume that $G$ is simply-connected. Let $\{ Z_1, \ldots , Z_n\}$ be a global unitary frame on $G$ that is $\nabla^s$-parallel. Since the connection $\nabla^s$ is left invariant, we know that for any $a\in G$, the pushed forward frame $dL_aZ$ by the left translation is also parallel, thus is a constant unitary change of $Z$. That is, for any $a\in G$ we have $P(a)\in U(n)$ such that
$$  dL_aZ=P(a)^{-1}Z.$$
Here we wrote $Z$ for the column vector $ ^t\!(Z_1, \ldots , Z_n)$. The map $P: G \rightarrow U(n)$ is a homomorphism, since for any $a,b\in G$, we have $L_{ab}=L_aL_b$, so
$$ P(ab)^{-1}Z = dL_{ab}Z = dL_a \,dL_b Z = dL_a((P(b)^{-1}Z) = P(b)^{-1} P(a)^{-1}Z,$$
hence $P(ab)=P(a)P(b)$. Now let $e= \,^t\!(e_1, \ldots , e_n)$ be the left invariant frame on $G$ which coincides with $Z$ at the identity element $o$ of $G$. Then we get $e(a)=P(a)^{-1}Z(a)$.  So the matrix of connection $\nabla^s$ under $e$ is given by
$\theta^s= dP^{-1}P = - P^{-1}dP$, namely
\begin{equation}
(\theta^s)_i^j = \sum_{k=1}^n \{  \Gamma^j_{ik}\varphi_k + \Gamma^j_{i\overline{k}} \overline{\varphi}_k \} ,  \label{P formula}
\end{equation}
where $\varphi$ is the coframe on $G$ dual to $e$, is given by
\begin{eqnarray}
\Gamma^j_{ik} & = &  \langle \nabla^s_{e_k}e_i , \overline{e}_j \rangle  \ = \ - \sum_{r=1}^n \overline{P_{ri}} e_k(P_{rj}) \\
\Gamma^j_{i\overline{k}} & = &  \langle \nabla^s_{\overline{e}_k}e_i , \overline{e}_j \rangle  \ = \ - \overline{ \Gamma^i_{jk} }
\end{eqnarray}
Now let us denote by $p: {\mathfrak g} \rightarrow {\mathfrak u}(n)$ the homomorphism on the Lie algebras induced by $P$, and extend it linearly over ${\mathbb C}$. We have
\begin{equation}
p([e_i,e_k]) = [ p(e_i), \,p(e_k)], \ \ \ \ p([e_i,\overline{e}_k]) = [ p(e_i), \,p(\overline{e}_k)] \label{little p}
\end{equation}
for any $i,k$.  Note that for any $X\in {\mathfrak g}^{\mathbb C}$, $p(X) = dP_{o}(X) = - \theta^s(X)$. So by our notations  $[e_i, e_k] = \sum_r C^r_{ik}e_r$, and
$ [\overline{e}_j,e_i] = \sum_r \big( D^j_{ri}\overline{e}_r - \overline{D^i_{rj}} e_r \big)$, where $C^j_{ik}$ and
$$ D^j_{ik}=\langle [\overline{e}_j,e_k], e_i\rangle = \Gamma^j_{ik} - sT^j_{ik}, $$
are the structure constants, we get the following

\begin{lemma} \label{algebriac relation}
Let $G$ be a Lie group equipped with a left invariant metric $g$ and a compatible left invariant complex structure $J$, and assume that the Hermitian manifold $(G,g)$ is $\nabla^s$-flat. Then under any left invariant unitary frame $e$, the components of the connection $\nabla^s$ satisfy
\begin{eqnarray}
&& \sum_{r=1}^n \big( C^r_{ik} \Gamma^{\ell}_{jr} + \Gamma^r_{ji} \Gamma^{\ell}_{rk} -  \Gamma^r_{jk} \Gamma^{\ell}_{ri} \big) \ = \ 0\\
&& \sum_{r=1}^n \big(  D^j_{ri} \overline{\Gamma^k_{\ell r} } + \Gamma^{\ell}_{kr} \overline{D^i_{rj} } + \Gamma^{\ell}_{ri} \overline{\Gamma^k_{rj} } -
 \Gamma^r_{ki} \overline{\Gamma^r_{\ell j} } \big) \ = \ 0
\end{eqnarray}
for any indices $i$, $j$, $k$ and $\ell$.
\end{lemma}

Note that both $C$ and $D$ can be expressed in terms of $T$ and $\Gamma$, and the covariant derivatives of $T$ can also be expressed in terms of $T$ and $\Gamma$. While we do believe  that the identities in Lemmas 2.1, 3.1, and 4.1 are so over determined that it will force $T$ and $\Gamma$ to vanish hence give an affirmative answer to Question 1.1, at this point we do not know how to decipher those algebraic relations between $T$ and $\Gamma$. So in this article we will only try a couple of special cases: when $n=2$ or when $\Gamma =0$.

\begin{proof}[Proof of Theorem \ref{intro1}]

Let $(G,g)$ be a Lie group of real dimension $4$ equipped with a left invariant metric and a compatible complex structure. Suppose that it has a flat $s\,$-Gauduchon connection $\nabla^s$ with $s\neq 0,2$. Assuming that $g$ is not K\"ahler, we want to derive a contradiction. Let $e$ be a left invariant unitary frame on $G$. Replacing it by a constant unitary change if necessary, we may assume that $T^2_{12}=0$ and $T^1_{12}=\lambda$. Note that the constant $\lambda \neq 0$ as otherwise $g$ would be K\"ahler. Under $e$, we have
\begin{eqnarray*}
&& T^1_{12,\ell} = -\lambda \Gamma^2_{2\ell}, \ \ \ T^2_{12,\ell} = \lambda \Gamma^2_{1\ell} \\
&& T^1_{12,\overline{\ell}} = \lambda \overline{\Gamma^2_{2\ell}}, \ \ \ T^2_{12,\overline{\ell}} = - \lambda \overline{\Gamma^1_{2\ell}}
\end{eqnarray*}
By the third identity in Lemma 3.2, we know that $s\neq 1, \frac{1}{2}$. Combine the above with Lemma 3.2, we get $\Gamma^1_{22}=\Gamma^2_{11} = \Gamma^2_{12} = \Gamma^2_{21} = 0$ and
$$\Gamma^2_{22}=(s-2)\lambda = \frac{s^2(s-2)}{4(s-1)(2s-1)}\lambda , \ \ \ \ \ \Gamma^1_{21}=- \frac{s(3s^2-8s+4))}{4(s-1)(2s-1)}\lambda  $$
Since $s\neq 2$, the first equality above gives $s^2=4(s-1)(2s-1)$, or equivalently, $7s^2-12s+4=0$. From that we solve for $s$ and get $s=\frac{2}{7}(3\pm \sqrt{2})$. To rule out these two special values, let us recall the relations
$$ D^j_{ik} = \Gamma^j_{ik} - sT^j_{ik}, \ \ \ \ C^j_{ik} = -D^j_{ik} + D^j_{ki} -2T^j_{ik}.$$
We get $ D^1_{22}=D^2_{11}=D^2_{12}=D^2_{21}=C^2_{12}=0$, and
$$ D^2_{22}=(s-2)\lambda , \ \  \ D^1_{21}=\Gamma^1_{21}+s\lambda = (5s-4)\lambda, $$
and $C^1_{12}+D^1_{12}=D^1_{21}-2\lambda = (5s-6)\lambda$. By considering the coefficient in front of $e_2$ for the Jacobi identity involving $\{ e_1, e_2, \overline{e}_1\}$, or equivalently, by letting $i=\ell=1$ and $j=k =2$ in (\ref{eq:thirdJacobi}), we get
$$ (C^1_{12}+D^1_{12})\overline{D^1_{21}} = D^1_{21} \overline{D^2_{22}} $$
That is, $(5s-6)\lambda (5s-4)\overline{\lambda} = (5s-4)\lambda (s-2)\overline{\lambda}$. Since $\lambda \neq 0$ and $5s-4\neq 0$, this implies that $5s-6=s-2$, or $s=1$, contradicting to the condition $s=\frac{2}{7}(3\pm \sqrt{2})$ that we obtained earlier. So when $n=2$ and $s\neq 0, 2$, any $(G,g)$ that is $\nabla^s$-flat must be K\"ahler. This completes the proof of Theorem 1.2.
\end{proof}

\begin{proof}[Proof of Theorem \ref{intro2}]

Let $G$ be a Lie group of real dimension $2n$ equipped with a left invariant metric $g$ and a compatible left invariant complex structure $J$, and suppose that $(G,g)$ is $\nabla^s$-flat with $s\neq 0,2$. Under the assumption of Theorem 1.3, the Lie group admits left invariant parallel frame $e$. This means that $\Gamma^j_{ik}=0$ for any $1\leq i,j,k\leq n$. Thus we have
$$ D=-sT, \ \ \ C=2(s-1)T.$$
The vanishing of $\Gamma$ also implies that, under the frame $e$, the components $T^j_{ik,\, \ell}=T^j_{ik, \, \overline{\ell}} =0$ for all indices. So Lemma 3.1 will take a particularly simple form. By letting $i=k$ and $j=\ell$ in the last equality of Lemma 3.1, we get
\begin{equation}
\sum_{r=1}^n \{   4(s-1)^2 | T^r_{ij}|^2 + (5s^2-10s+4)  ( T^i_{ir} \overline{T^j_{jr}} - |T^i_{jr}|^2 )  - s^2 ( |T^j_{ir}|^2 -  T^j_{jr} \overline{T^i_{ir}} ) \} = 0. \label{eq:2n}
\end{equation}
Since we already proved Theorem 1.2, we may assume that $n\geq 3$. Now let us separate our discussion into two cases: $s\neq 1$ and $s=1$. First assume that $s\neq 1$. By the third equality in Lemma 3.1, we know that $\sum_r T^{\ell}_{ir} T^r_{jk} = 0$. Since $C=2(s-1)T$, this means that
$$[X,[Y,Z]]=0  \ \ \ \ \forall  \ X,Y,Z\in {\mathfrak g}^{1,0}, $$
that is,  ${\mathfrak g}^{1,0}$ is $2$-step nilpotent.  Let $V_0$ be the center of ${\mathfrak g}^{1,0}$ and we choose $e$ so that $\{ e_1, \ldots , e_p\}$ spans $V_0$. Then we know that
$$ T^j_{\ast \ast }= T^{\ast }_{i\ast } = 0$$
for any $1\leq i\leq p$ and any $p<j\leq n$. Here $\ast$ can be any integer between $1$ and $n$. By letting $1\leq j\leq p$ in (\ref{eq:2n}), we get $\sum_r |T^j_{ir}|^2=0$, that is, $T^j_{\ast \ast }=0$ for any $1\leq j\leq p$ also. So $T=0$ and $g$ is K\"ahler.

\vsv

Now we assume that $s=1$. In this case we have $C=0$ and $D=-T$. The first equality of Lemma 3.1 now says that
\begin{equation}
 \sum_{r=1}^n   \{  T^{\ell}_{jr}  T^r_{ik} -   T^{\ell}_{kr} T^r_{ij}  \} =0 , \label{eq:A}
\end{equation}
which is just the Jacobi identity for $\{ e_j, e_k, \overline{e}_{\ell} \}$ as $C=0$ in this case. Let $s=1$ in (\ref{eq:2n}) and taking the real part, we get
\begin{equation}
 \sum_{r=1}^n \{ |T^j_{ri}|^2 - |T^i_{rj}|^2\} = 0  \label{eq:B}
 \end{equation}
for any $1\leq i, j\leq n$. We want to conclude that $T$ must vanish in this case. Following \cite{YZ}, for any $X=\sum X_ie_i$ in $V={\mathfrak g}^{1,0}$, let us denote by $A_X$ the linear transformation on the vector space $V$ given by $A_X(e_j) = \sum_{i,k} X_iT^k_{ij}e_k$. Multiplying on equation (\ref{eq:A}) by $X_iY_k$ and add up, and then by $X_kY_i$ and add up, we see that $A_XA_Y = - A_Y A_X$ for any $X$, $Y$ in $V$. By the {\em Claim} in the proof of Theorem 2 of \cite{YZ}, we know that there exists non-zero vector $W\in V$ such that $A_Y(W)=0$ for any $Y\in V$. Without loss of generality, we may assume that $W=e_n$. So we have $T^i_{nk}=0$ for any $i$, $k$. By letting $j=n$ in (\ref{eq:B}), we see that $T^n_{ik}=0$ for any $i$, $k$. Restricting $T$ to the subspace $V_1$ spanned by $\{ e_1, \ldots , e_{n-1}\}$ and repeating the argument, we eventually get $T=0$ for all indices. That is, the metric $g$ is K\"ahler in the $s=1$ case as well. This concludes the proof of Theorem 1.3.
\end{proof}

\appendix

\section{Lie groups with a flat left-invariant K\"ahler structure}

In this section we recall some results on flat left-invariant Lie groups due to Milnor \cite{Milnor}
and Barberis-Dotti-Fino \cite{BDF}. These results lead to Corollary \ref{flat Kahler intro} stated in
Section 1.

\begin{theorem}[Theorem 1.5 on p.298 in \cite{Milnor}]
A connected Lie group $G$ with a left invariant Riemannian metric is flat if and only if
its Lie algebra $\mathfrak{g}$ admits an orthogonal decomposition
\begin{equation}
\mathfrak{g}=\mathfrak{h} \oplus \mathfrak {i},
\end{equation} where $\mathfrak{h}$ is an abelian Lie subalgebra and $\mathfrak{i}$ is an abelian
ideal. Moreover, $\operatorname{ad}(X)$ is skew-adjoint for any $X \in \mathfrak{h}$.
\end{theorem}

In the above theorem, the abelian ideal $\mathfrak i$  is $\mathfrak{i}=\{Y \in \mathfrak{g}\ |\ \nabla_{Y}=0 \}$, where
$\nabla$ is the Levi-Civita connection. It is further observed (Proposition 2.1 in \cite{BDF}) that
$\mathfrak{i}=\mathfrak{c} \oplus [\mathfrak{g}, \mathfrak{g}]$ where $\mathfrak{c}$ is the center of
$\mathfrak{g}$ and $[\mathfrak{g}, \mathfrak{g}]$ is of even dimension.

Lie groups with a flat left-invariant K\"ahler structure were classified in Proposition 2.1 and Corollary 2.2 in \cite{BDF}. We summarize and slightly rephrase their results as follows.

\begin{theorem}[Proposition 2.1, Corollary 2.2, and Proposition 3.1 in \cite{BDF}] \label{flat k}
Let $(G, J, g)$ be connected Lie group with a left invariant complex structure and a flat left invariant
K\"ahler metric. Then the Lie algebra $\mathfrak{g}$ of $G$  admits an orthogonal decomposition
\begin{equation}
\mathfrak{g}=\mathfrak{h} \oplus \mathfrak {c} \oplus [\mathfrak{g}, \mathfrak{g}],
\end{equation}
which satisfies the following properties:
\begin{enumerate}
\item $\mathfrak{h}$ is a Lie subalgebra, $\mathfrak {c}$ is the center of $\mathfrak {g}$, and $[\mathfrak{g}, \mathfrak{g}]$ is $J$-invariant.

\item  The adjoint action $\operatorname{ad}$ injects $\mathfrak{h}$ to skew-symmetric transformations on $[\mathfrak{g}, \mathfrak{g}]$ which commutes with $J$. Furthermore, we have the `simultaneous diagonalization' of the restriction of $\operatorname{ad}_{X}=\nabla_X$ on $\mathfrak{h}$ in the following sense:

     There exists an orthonormal basis with $[\mathfrak{g}, \mathfrak{g}]=\operatorname{span} \{e_1, Je_1, \cdots, e_p, Je_p\}$ and an injective linear map $q=(q_1, \cdots, q_p): \mathfrak{h} \rightarrow \mathbb{R}^p$ such that for any $X \in \mathfrak{h}$ and $1 \leq i \leq p$, $\operatorname{ad}_{X}$ has the block diagonal form:
\begin{equation}
ad_X \left[ \begin{array} {cc} e_i  \\ Je_i  \end{array} \right]=
\left[ \begin{array} {cc} 0 & q_i(X)  \\ -q_i (X) & 0 \end{array} \right]
\left[ \begin{array} {cc} e_i  \\ Je_i  \end{array} \right].\nonumber
\end{equation}

\item There are further orthogonal decompositions $\mathfrak{h}=(\mathfrak{h} \cap J\mathfrak{h}) \oplus \mathfrak{h}_1$ and $\mathfrak{c}=(\mathfrak{c} \cap J\mathfrak{c}) \oplus \mathfrak{c}_1$, and
   $J$ restricts to an isomorphism between $\mathfrak{h}_1$ and $\mathfrak{c}_1$.
\end{enumerate}
\end{theorem}

It is straightforward to write down all such Lie algebras in real dimension 4, which is exactly covered in
Corollary \ref{flat Kahler intro}.

It might be interesting to compare the kernel space $\{Y \in \mathfrak{g}\ |\ \nabla_{Y}=0 \}$ in Theorem \ref{flat k} with the Lie algebra homomorphism $p: {\mathfrak g} \rightarrow {\mathfrak u}(n)$ defined in (\ref{little p}) in Section 4. Let $\nabla=\nabla^{s}$, then our Theorem \ref{intro2} can be reinterpreted as: for a $\nabla^{s}$-flat Hermitian Lie group $(G, J, g)$ with a left invariant structure, if the kernel space $\{Y \in \mathfrak{g}\ |\ \nabla_{Y}=0 \}=\mathfrak{g}$, then $(G, J, g)$ is K\"ahler and isomorphic to a complex vector group. We will study the structure of a $\nabla^{s}$-flat left invariant Hermitian Lie group in the spirit of Theorem \ref{flat k} in a future work.

\end{document}